\documentclass[11pt]{article}

\usepackage{amsfonts, amssymb}
\usepackage{amsmath}
\usepackage{amsthm}
\usepackage{color}
\usepackage{indentfirst}
\usepackage{verbatim}   
\usepackage{mathrsfs}   

\usepackage{enumerate}

\usepackage{indentfirst}
\usepackage{color}

\renewcommand{\L}{\mathcal{L}}
\newcommand{\R}{\mathcal{R}}
\newcommand{\D}{\mathcal{D}}
\renewcommand{\H}{\mathcal{H}}
\newcommand{\J}{\mathcal{J}}
\newcommand{\I}{\mathcal{I}}

\newcommand{\p}{\prime}
\newcommand{\n}{\mathbf{n}}
\newcommand{\m}{\mathbf{m}}
\newcommand{\pt}{{\cal PT}_n}
\renewcommand{\t}{{\cal T}_n}
\newcommand{\w}{{W}}

\newcommand{\s}{{\cal S}_n}

\newcommand{\sr}{{SR}_n}

\newcommand{\oor}{{OR}_n}

\newcommand{\rk}{\mathrm{rk}}
\newcommand{\tp}{\mathrm{tp}}
\newcommand{\im}{\mathrm{im}}
\newcommand{\dom}{\mathrm{dom}}

\newcommand{\qaq}{\quad\text{and}\quad}

\begin{document}
\newtheorem{lemma}{Lemma}[section]
\newtheorem{proposition}[lemma]{Proposition}
\newtheorem{theorem}[lemma]{Theorem}
\newtheorem{corollary}[lemma]{Corollary}
\newtheorem{definition}[lemma]{Definition}
\newtheorem{example}[lemma]{Example}
\newtheorem{remark}[lemma]{Remark}
\numberwithin{equation}{section}

\title{Congruences on Orthogonal Rook Monoids and\\ Symplectic Rook Monoids}
\date{}

\author {Jianqiang Feng  ~~Zhenheng Li}

\vspace{ 0.2cm}

\maketitle

\begin{abstract}

We give a complete classification of all nonuniform congruences on orthogonal rook monoids and symplectic rook monoids. We find that there are four kinds of nonuniform congruences on the orthogonal rook monoids $\oor$ for even $n\ne 4$, and we describe each kind of the congruences explicitly in terms of normal subgroups of maximal subgroups. We also find that if $n = 4$, there are six kinds of nonuniform congruences on ${OR}_4$, and we describe these congruences using both $\mathcal{H}$-relations and certain normal subgroups of some maximal subgroups. In contrast, we find that there is only one kind of congruences on the symplectic rook monoids for all even $n\ge 2.$

\vspace{ 0.3cm} \noindent {\bf Keywords:} Congruence, orthogonal rook monoid, symplectic rook monoid, orthogonal Weyl group, symplectic Weyl group, admissible set.

\vspace{ 0.3cm} \noindent {\bf 2010 AMS Subject Classification:} 20M32, 20M10

\end{abstract}

\baselineskip 14pt
\parskip 1mm
\section{Introduction}
Orthogonal rook monoids and symplectic rook monoids are submonoids of rook monoids. Known respectively as orthogonal Renner monoids and symplectic Renner monoids in the theory of linear algebraic monoids \cite{P1, R2, S1}, they first appeared in 2001. Their structures are described explicitly in \cite{ZhenhengLi2001, ZhenhengLi2003, LR03} using admissible subsets and elementary matrices.


In this paper we are interested in the congruences on the orthogonal rook monoids and symplectic rook monoids. Before describing our main results, we provide some historical information about the congruences on some commonly encountered monoids.

The congruences on the monoid $\pt$ of all partial transformations were characterized in \cite{Sut}; the congruences on the monoid $\t$ of all full transformations were investigated in \cite{M1} whose results are then used to describe the congruences on the rook monoid $\mathcal{R}_n$ in \cite{Lib}. A detailed description of the congruences on $\pt, \t, \mathcal{R}_n$ can be found in \cite{GM}, and a systematic description of the congruences on their direct products are given in \cite{ABG}. There are elegant results on the congruences on inverse semigroups and regular semigroups; we refer the reader to \cite{Fe, PM, Pe, S, T} and the references therein for more details.

Let $\s$ be the symmetric group on $\n = \{1, \ldots, n\}$.
A submonoid $M$ of the monoid $\pt$ is called $\s$-normal if $s^{-1}\sigma s\in M$
for all $\sigma\in M$ and $s\in\s$. The monoid $\pt$ itself is $\s$-normal. The monoid $\t$ and the rook monoid $\mathcal R_n$ are also $\s$-normal.

The congruences on $\s$-normal monoids were described by Levi in \cite{Le}, from which it follows that if $n\ge 5$ then all nontrivial congruences on almost all $\s$-normal submonoids $M$ of $\pt$ have the form $\equiv_N$ where $N$ is a normal subgroup of ${S}_k, k = 1, \ldots, n$. For any $\sigma, \tau\in M$, the congruence $\sigma\equiv_N\tau$ is defined as follows.
\begin{enumerate}[{\rm\indent 1)}]
  \vspace{-2mm}
  \item \indent If $\rk(\sigma) < k$, then $\sigma\equiv_N\tau$ if and only if $\rk(\tau)< k$.
  \vspace{-2mm}
  \item \indent If $\rk(\sigma) > k$, then $\sigma\equiv_N\tau$ if and only if $\sigma = \tau$.
  \vspace{-2mm}
  \item \indent If $\rk(\sigma) = k$, then $\sigma\equiv_N\tau$ if and only if $\sigma$ and $\tau$ are $\mathcal{H}$-related and there exists $\mu\in N$ such that $\tau=\mu(\sigma)$ (see (\ref{hclass}) for more details).
\end{enumerate}

The orthogonal rook monoids $\oor$ and symplectic rook monoids $\sr$ are not $\s$-normal for $n=2m\geq 6$. Indeed, the unit group of $\sr$ is
\begin{align*}
    \w &=\{\tau\in S_n\mid\tau(\theta(i)) =\theta(\tau(i)) \mbox{ for all } i\in\n\}~,
\end{align*}
where $\theta$ is the involution of $\n$ defined by
$
    \theta(i) = n+1-i.
$
Take
$$\sigma=\left(\begin{array}{ccccccccc}
              1 & 2 & 3 & 4 & \dots & n-3 & n-2 & n-1 & n \\
              3 & 1 & 2 & 4 & \dots & n-3 & n-1 & n & n-2
            \end{array}\right)\in\oor\subset\sr
$$
$$s=\left(\begin{array}{ccccccccc}
              1 & 2 & 3   & 4 & \dots & n-3 & n-2 & n-1 & n \\
              2 & n & n-1 & 4 & \dots & n-3 & 3   & n-2 & 1
\end{array}\right)\in\s~.
$$
We find $s^{-1}\sigma s\notin\w$ since $s^{-1}\sigma s (\theta(1)) \ne \theta(s^{-1}\sigma s (1))$. Therefore, $s^{-1}\sigma s\notin\sr$.

We ask naturally. Are the congruences on the orthogonal rook monoids and symplectic rook monoid dramatically different from those on $\s$-normal monoids? If yes, how to determine them?

The above questions motivate us to investigate the congruences on the orthogonal rook monoids and the symplectic rook monoids. We find that though the orthogonal rook monoids $\oor$ and symplectic rook monoids $\sr$ are not $\s$-normal for $n=2m\geq 6$, the congruences on the orthogonal rook monoids are a lot different from those on $\s$-normal monoids, but the congruences on the symplectic rook monoids are consistent with those on $\s$-normal monoids.

In this paper we introduce four kinds of nonuniform congruences: $\equiv_N,$ $ \equiv_{N_1,N_2}, \equiv_N^{\rm I},$ and $\equiv_N^{\rm II}$ on the orthogonal rook monoids $\oor$ with $n$ even, and then show that they are the complete set of all nonuniform congruences on the orthogonal rook monoids for $n\ne 4$. For $n=4$, in addition to the above four kinks of congruences, we define two more kinds of congruences: $\equiv_1, \equiv_2$ on ${OR}_4$, and prove that these six kinds of congruences are the complete set of all nonuniform congruences on ${OR}_4$. It is important to highlight that the congruences $\equiv_1$ and $\equiv_2$ or the like don't exist on $\oor$ for $n\ne 4.$

The following is a summary of our main results about the congruence on $\oor$.
\begin{theorem}
Let $\rho$ be a nonuniform congruence on $\oor$. If $n\ne 4$, then
$\rho$ coincides with $\equiv_N$ for some $N\unlhd {S}_k, k\in\{1,\dots,m-1\}$,
or $\rho$ coincides with $\equiv_{N_1,N_2}$ for some $N_1,N_2\unlhd{S}_m$,
or $\rho$ coincides with $\equiv_N^{\rm I}$ or $\equiv_N^{\rm II}$ for some $N\unlhd
{S}_m$.
If $n = 4$, in addition to the above possible congruences, $\rho$ can also be $\equiv_1$ or $\equiv_2$.
\end{theorem}

Contrary to the case of the orthogonal rook monoids, we find that there is only one kind of congruences $\equiv_N$ on the symplectic rook monoids, given in Theorem \ref{symConResult}.

To obtain the main results we first gather in Section \ref{pre} some basic notations and facts about orthogonal rook monoids and symplectic rook monoids, and then investigate in Section \ref{green relations} the ideals and Green relations on $\oor$. We then describe in Section \ref{conOrthogonal} the congruences on the orthogonal rook monoids. We at last describe in Section \ref{conSymplectic} the congruences on the symplectic rook monoids.

From now on, we always assume that $n$ is a positive integer of the form $n=2m\ge 2$.

{\bf Acknowledgement } The authors thank Reginald Koo for useful comments on the manuscript.

\section{Preliminaries}\label{pre}

The rook monoid $\mathcal{R}_n$ consists of all injective partial mappings $\sigma$ from a subset of $\n$ to a subset of $\n$. We use $I(\sigma)$ to denote the domain of $\sigma$, and $J(\sigma)$ the image of $\sigma$. The cardinality $|I(\sigma)|$ is called the {\it rank} of $\sigma$, and is denoted by $\mathrm{rk}(\sigma)$.

To define the orthogonal rook monoid and symplectic rook monoids, we need the involution $\theta$ of $\n$ defined by
$$
    \theta(i) =n+1-i~.
$$
A proper subset $I$ of $\n$ is {\em admissible} if the intersection of $I$ and $\theta(I) $ is empty; the set $\n$ and the empty set $\emptyset$ are considered admissible in order to make mathematical expressions simpler. We often write $\bar i = \theta(i)$ for simplicity.

An admissible subset $I$ is referred to as an admissible $k$-subset if $|I| = k$. Clearly, there is no admissible $k$-subset for $k=m+1, \ldots, n-1.$ 

In what follows when mentioning admissible subsets, we always mean admissible subsets of $\n$.

An injective partial mapping $\sigma$ of $\n$ is {\em symplectic}  if both its domain and image are proper admissible subsets of $\n$, or if $\sigma$ is of rank $n$ and it transforms any admissible subset of $\n$ to an admissible subset.
The {\em symplectic rook monoid}, denoted by $\sr$, is the monoid consisting of all symplectic injective partial mappings from $\mathbf{n}$ to $\mathbf{n}$, with multiplication the composition of
partial mappings, identity element the identity mapping, and zero element the empty mapping.

The unit group of $\sr$ is
\vspace{-2mm}
\begin{align*}
    \w & = \{\sigma\in\mathcal{R}_n\mid \rk(\sigma) = n \text{ and } \sigma\text{ maps admissible sets to admissible sets}\} \\
        &=\{\sigma\in S_n\mid\sigma(\bar{i}) =\overline{\sigma(i)} \mbox{ for all } i\in\n\}~ .
\vspace{-3mm}
\end{align*}
The group $\w$ is isomorphic to the symplectic Weyl group.
Define
\begin{equation}\label{wprime}
  \w' = \{\sigma\in\w ~\big |\,~ |\sigma(\m) \cap (\n\setminus\m)| \text{ is even }\}~,
\end{equation}
where $\m=\{1, 2, \ldots, m\}$. Then $\w'$ is a normal subgroup of $\w$, and $\w'$ is isomorphic to the even special orthogonal Weyl group.

\begin{definition}
    An admissible $m$-subset $A$ is of type {\rm I} if there exists $w\in\w'$ such that $wA = \{1, \ldots, m\}$; type {\rm II} if $wA = \{1, \ldots, m-1, m+1\}$.
\end{definition}
\vspace{-3mm}

It is easily seen that an admissible $m$-subset $A$ is of type {\rm I} if and only if it contains an even number of elements greater than $m$; type {\rm II} if and only $A$ contains an odd number of elements greater than $m$.
For instance, if $n=2m=8$ then the sets $\{1, 2, 3\}$ and $\{2, 5, 6\}$ are admissible $3$-subsets of type {\rm I}, but $\{1, 2, 5\}$ and $\{5, 6, 7\}$ are admissible $3$-subsets of type {\rm II}. It is obvious that an admissible $m$-subset is of either type {\rm I} or else type {\rm II}.

An injective partial mapping $\sigma$ of $\n$ is {\em orthogonal} if either
\begin{enumerate}[{\rm\indent 1)}]
\vspace{-3mm}
  \item $\rk(\sigma)<m$ and both $I(\sigma)$ and $J(\sigma)$ are admissible, or
\vspace{-3mm}
  \item $\rk(\sigma)=m$ and both $I(\sigma)$ and $J(\sigma)$ are admissible and of the same type, or
\vspace{-3mm}
  \item $\sigma\in\w'$.
  \vspace{-3mm}
\end{enumerate}
The {\em orthogonal rook monoid} $\oor$ consists of all orthogonal injective partial mappings from $\mathbf{n}$ to $\mathbf{n}$. Clearly, $\oor$ is a submonoid of $\sr$. The group $\w'$ is the unit group of ${\cal OR}_n$.
If $\alpha\in\oor$ and $\rk(\alpha) = m$, we define the type $\tp(\alpha)$ of $\alpha$ to be the type of its domain, which is equal to the type of its image.
We list two easy lemmas.
\begin{lemma}\label{rank}
For all $\sigma,\tau\in \sr$, the rank $\rk(\tau\sigma) \leq
\min(\rk(\sigma) ,\rk(\tau) ) $.
\end{lemma}
\begin{lemma}\label{idempotent}
The element $\alpha\in\sr$ is an idempotent if and only if $\alpha$
is the identity transformation of some admissible subset
$A\subseteq\n$.
\end{lemma}

For an admissible subset $A\subseteq\n$ we denote by $\varepsilon_A$ the
unique idempotent of $\oor$ such that $\dom(\varepsilon_A)=A$.

\section{Ideals and Green's relations in $\oor$}\label{green relations}

\begin{proposition}\label{principal}
Let $S=\oor$ and $\sigma\in S$.

{\rm i)} $\sigma S=\{\tau\in S\mid  J(\tau)\subseteq   J(\sigma)\}.$

{\rm ii)} $S\sigma=\{\tau\in S\mid  I(\tau)\subseteq   I(\sigma)\}.$

{\rm iii)} If $\rk(\sigma)\ne m$, then $S\sigma S =\{\tau\in S\mid\rk(\tau)\leq \rk(\sigma)\}$. If $\rk(\sigma) = m$, then
\[
    S\sigma S =\{\tau\in S \mid \rk(\tau)<\rk(\sigma),\text{ or } \tp(\sigma)=\tp(\tau) \text{ if }\rk(\tau)=m\} .
\]

\end{proposition}

\begin{proof}

{\rm i}) It is obvious that if $\gamma\in\oor$, then $ J(\sigma\gamma)\subseteq
 J(\sigma)$. That is,
\[
    \sigma S\subseteq\{\tau\in S\mid  J(\tau)\subseteq   J(\sigma)\}.
\]
Conversely, suppose $\tau\in S$ with $ J(\tau)\subseteq   J(\sigma)$. Let
$ J(\tau)=\{c_1,\dots,c_s\}$ and $ J(\sigma)=\{c_1,\dots,c_s,\dots,c_t\}$ where $s, t\in\{1, \ldots, m, n\}$ with $s\le t={\rm rk}(\sigma)$.
Set
$$\tau=\left(
\begin{array}{ccc}
  b_1 & \dots & b_s \\
  c_1 & \dots & c_s
\end{array}\right) \quad\text{and}\quad
\sigma=\left(
\begin{array}{ccccc}
  a_1 & \dots & a_s & \dots & a_t\\
  c_1 & \dots & c_s & \dots & c_t
\end{array}\right).
$$
Define $\gamma= \left(
\begin{array}{ccc}
  b_1 & \dots & b_s \\
  a_1 & \dots & a_s
\end{array}\right)$.
If $s< m$, the set $\{b_1,\dots,b_s\}$ is admissible since it is the domain of $\tau$; the set $\{a_1,\dots,a_s\}$ is also admissible since it is the pre-image of $ J(\tau)$ under $\sigma$. Thus $\gamma\in S$, and hence $\tau=\sigma\gamma\in\sigma S.$
If $s=n$ then $\sigma,\tau\in\w^\p$ and $\gamma=\sigma^{-1}\tau\in\w^\p$, so
$\tau=\sigma\gamma\in\sigma S.$

If $s=m$, by definition $\{b_1,\dots,b_s\}$ and $\{c_1,\dots,c_s\}$ are of the same type since they are the domain and image of $\tau$, respectively. Similarly, if $t=m$, then $\{a_1,\dots,a_t\}$ and $\{c_1,\dots,c_t\}$ are of the same type.  If $t=n$, then
$\{a_1,\dots,a_s\}$ and $\{c_1,\dots,c_s\}$ are also of the same type
since every element of $\w^\p$ keeps the type of each admissible
$m$-subset of $\n$. It follows that $\tau=\sigma\gamma\in\sigma S.$

ii) The proof is parallel to that of i).

iii) We now prove the first part of iii). If $\rk(\sigma)\ne m$, let
\[
    X=\{\tau\in S\mid\rk(\tau)\leq \rk(\sigma)\}.
\]
We obtain from Lemma \ref{rank} that
$S\sigma S\subseteq X$. Conversely, for each $\tau\in X$, let $$\tau=\left(
\begin{array}{ccc}
  c_1 & \dots & c_s \\
  d_1 & \dots & d_s
\end{array}\right)\quad\text{and}\quad\sigma=\left(
\begin{array}{ccc}
  a_1 & \dots & a_t \\
  b_1 & \dots & b_t
\end{array}\right),\quad\text{where } s=\rk(\tau)\leq t={\rm rk}(\sigma).
$$
Take any $s$-subset $\{a_{i_1},\dots,a_{i_s}\}$ from the admissible set $\{a_1,\dots,a_t\}$. If $t\leq m-1$, then $\{a_{i_1},\dots,a_{i_s}\}$ and $\{b_{i_1},\dots,b_{i_s}\}$ are both $s$-admissible subsets. Define
$$\delta= \left(
\begin{array}{ccc}
  b_{i_1} & \dots & b_{i_s} \\
  d_1 & \dots & d_s
\end{array}\right)\quad\text{and}\quad
\gamma= \left(
\begin{array}{ccc}
   c_1 & \dots & c_s \\
   a_{i_1} & \dots & a_{i_s}
\end{array}\right).$$
Then $\delta,\gamma\in S$ and
$\tau=\delta\sigma\gamma\in S\sigma S$.
If $t=n$, then $\sigma\in{\cal{W}}'$. If $s=m$, we require the type of the set  $\{a_{i_1},\dots,a_{i_s}\}$ chosen above to have the same type of $\tau$; this is achievable since $\n$ has admissible subsets of both types. It follows that $\delta$ and $\gamma$ are both in $ S$, and $\tau=\delta\sigma\gamma\in S\sigma S$.

We next show the second part of iii).  If $\rk(\sigma) = m$, let
\[
    X=\{\tau\in S \mid \rk(\tau)<m, \text{ or } \rk(\tau)=m\text{ with }\tp(\sigma)=\tp(\tau)\}.
\]
Let $\delta, \gamma\in S$. Clearly if $\rk(\delta\sigma\gamma)<m$, then $\delta\sigma\gamma\in X.$
If $\rk(\delta\sigma\gamma)=m$, from Lemma \ref{rank} it follows that $\rk(\delta), \rk(\gamma) \ge m$. Hence
$\sigma$ and $\delta\sigma\gamma$ are of the same type because the elements of $S$ of ranks $m$ and $n$ preserve the type of an admissible $m$-subset. We obtain $S\sigma S\subseteq X$.

Conversely, for each $\tau\in X$, if $\rk(\tau)<m$, a similar argument to part i) yields $\tau\in S\sigma S$. If $\rk(\tau)=m$ and $\tp(\tau)=\tp(\sigma)$, a similar argument to part i) again shows that there exist $\delta, \gamma\in\sr$ such that $\tau=\delta\sigma\gamma$. It is obvious that both $\delta$ and $\tau$ are of the same type as that of $\sigma$. We then find that $\delta$ and $\gamma$ lie in $S$ and thus $\tau\in S\sigma S$.
\end{proof}

For $k=0,1,\dots,m-1,n$ we define
\[
    \mathcal{I}_k=\{\sigma\in\oor\mid\rk(\sigma)\leq k\}~.
\]
We also define
\begin{align*}
  \mathcal{I}_m^{\rm I}    &= \{\sigma\in\oor\mid\rk(\sigma)<m; \text{ or }\rk(\sigma)=m,\tp(\sigma)={\rm I}\}~, \\
  \mathcal{I}_m^{\rm II} &= \{\sigma\in\oor\mid\rk(\sigma)<m; \text{ or }\rk(\sigma)=m,\tp(\sigma)={\rm II}\}~.
\end{align*}
In particular, $\mathcal{I}_n = \oor$ and $\mathcal{I}_m^{\rm I}\cup\mathcal{I}_m^{\rm II} = \oor\setminus\w'$.

The following result is a consequence of Proposition \ref{principal} iii).

\begin{corollary}\label{ideal}
    The sets $\mathcal{I}_m^{\rm I}, \mathcal{I}_m^{\rm II}$, and $\mathcal{I}_k$ are the only ideals of the orthogonal rook monoid $\oor$ where $k=0,1,\dots,m-1,n$.
\end{corollary}

From Corollary \ref{ideal} we find $\oor$ has two chains of ideals with respect to inclusion:
\[
    \I_0\subseteq\I_1\subseteq\dots\subseteq\I_{m-1}\subseteq\I_m^*\subseteq\I_n~,
\]
where $\I_m^*$ is either $\I_m^{\rm I}$ or $\I_m^{\rm II}$.

We use $\mathcal{L, R, H, J, D}$ to denote the Green relations on $\oor$, and for each $\sigma\in\oor$ we use $\mathcal{L(\sigma), R(\sigma), H(\sigma), J(\sigma), D(\sigma)}$ to denote the corresponding equivalence classes, respectively.
\begin{proposition}\label{green} Let $\sigma\in\oor$. Then
\begin{enumerate}[{\rm i)}]\vspace{-2mm}
    \item $\mathcal{L}(\sigma) = \{\tau\in\oor \mid I(\tau) = I(\sigma)\}$. \vspace{-2mm}

    \item  $\mathcal{R}(\sigma) = \{\tau\in\oor \mid J(\tau) = J(\sigma)\}$. \vspace{-2mm}

    \item  $\mathcal{H}(\sigma) = \{\tau\in\oor \mid I(\tau) = I(\sigma) \text{ and } J(\tau) = J(\sigma)\}$. \vspace{-2mm}

    \item  $\mathcal{J}(\sigma) = \{\tau\in\oor \mid \rk(\tau) = \rk(\sigma), \text{ and } \tp(\tau) = \tp(\sigma) \text{ if } \rk(\sigma) =m\}$. \vspace{-2mm}

    \item  $\mathcal{D}(\sigma) = \mathcal{J}(\sigma)$. \vspace{-2mm}
\end{enumerate}
\end{proposition}
\begin{proof}
The statements i) and ii) follow from Proposition \ref{principal} i) and ii), respectively.
The statement iii) is a consequence of i) and ii) since $\H=\R\cap \L$. The statement (iv) follows from Proposition \ref{principal} iii). The statement (v) follows from the fact that $\oor$ is a finite monoid.
\end{proof}

\begin{proposition}\label{greenorthnum} Let $\sigma\in\oor$ and $k=\rk(\sigma)$. Then
\begin{enumerate}[{\rm i)}]\vspace{-2mm}
    \item $|\mathcal{L}(\sigma)| =|\mathcal{R}(\sigma)| = \begin{cases}
                                    {m\choose k}2^kk!   &   \mbox{if } k =0, 1, \ldots, m-1\\
                                    2^{m-1}m!   &   \mbox{if } k =m,n~.
                                    \end{cases}$
    \vspace{-2mm}
    \item The number of $\L$-classes equals that of $\R$-classes, and is $1+\sum_{k=0}^m{m \choose k}2^k$~.
    \vspace{-2mm}
    \item  $|\mathcal{H}(\sigma)| = \begin{cases}
                                    k!   &   \mbox{if } k =0, 1, \ldots, m\\
                                    2^{m-1}m!    &   \mbox{if } k = n~.
                                    \end{cases}  $
    \vspace{-2mm}
    \item The number of $\H$-classes is $1+\frac{1}{2}\cdot 4^m+\sum_{k=0}^{m-1}{m \choose k}^24^k$~.
    \vspace{-2mm}
    \item  $|\mathcal{D}(\sigma)| =|\mathcal{J}(\sigma)| = \begin{cases}
                                    {m\choose k}^24^kk!   &   \mbox{if } k =0, 1, \ldots, m-1\\
                                    \frac{1}{2}\cdot 4^m m!  &   \mbox{if } k =m\\
                                    2^{m-1}m!    &   \mbox{if } k = n~.
                                    \end{cases}$
    \vspace{-2mm}
    \item The number of $\J$-classes equals that of $\D$-classes, and is $m+3$~.
\end{enumerate}
\end{proposition}
\begin{proof}

We prove i) and ii) for $\L$-classes only since the proof for $\R$-classes is similar. From Proposition \ref{green}
i) it follows that the $\L$-class containing $\sigma$ is uniquely determined by $I(\sigma)$, which must be an admissible subset of $\n$.
For each $k=0,1,\dots,m-1$, the number of
$\L$-classes is ${m \choose k}2^k$ and each $\L$-class consists of
${m \choose k}2^kk!$ elements of rank $k$. For $k=m$, the number of
$\L$-classes is $2^m$ and each $\L$-class consists of $2^{m-1}m!$
elements of rank $m$. For $k=n$ the number of $\L$-classes is $1$
and the $\L$-class is the orthogonal Weyl group $\w^\p$ consisting
of $2^{m-1}m!$ elements.

We now prove iii) and iv). From Proposition \ref{green} iii) it follows that the $\H$-class
containing $\sigma$ is uniquely determined by $I(\sigma)$ and $J(\sigma)$: they must be admissible subsets of $\n$, and in addition are of the same type when $|I(\sigma)|=m$.
For each $k=0,1,\dots,m-1$, the number of
$\H$-classes is ${m \choose k}^24^k$ and each $\H$-class consists of
$k!$ elements of rank $k$. For $k=m$, the number of  $\H$-classes is
$\frac{1}{2}\cdot 4^m$ and each $\H$-class consists of $m!$ elements of
rank $m$. For $k=n$ the number of $\H$-classes is $1$ and the
$\H$-class is $\w^\p$.

We next show v) and vi). From Proposition \ref{green} iv) it follows that the
$\J$-class (resp. $\D$-classes) containing $\sigma$ is uniquely determined by
the rank of $\sigma$, the domain and image being admissible subsets of $\n$, and of the same
type when the rank of $\sigma$ is $m$. There are ${m\choose
k}^24^kk!$ elements of rank $k$ for $k=0,1,\dots,m-1$, and there are
$\frac{1}{2}\cdot 4^m m!$ elements of rank $m$ with the same type of
domain and image. For $k=n$ the number of $\J$-classes is $1$ and
the $\J$-class is $\w^\p$, with $|\w'| = 2^{m-1}m!$.
\end{proof}

\begin{proposition}\label{hclassgroup}
Let $H$ be an $\H$-class of $\oor$. If $H$ contains an idempotent of rank $k$, then $H = \w'$ for $k=n$, or $H$ is isomorphic to the symmetric group ${S}_k$ for $k=1,\dots,m$.
\end{proposition}
\begin{proof}
If $k=n$, it is obvious that $H$ is just the unit group of $\oor$ and hence $H=\w^\p$. If $k=1,\dots,m$, let $\varepsilon_A\in H$ be an idempotent of $H$ where $A=\{a_1,\dots,a_k\}\subseteq\n$. By Proposition
\ref{green} (iii) we find $\sigma\in H$ if and only if $\sigma$ has the form
$$\sigma=\left(
\begin{array}{cccc}
  a_1 & a_2 & \dots & a_k \\
  a_{\mu(1)} & a_{\mu(2)} & \dots & a_{\mu(k)}
\end{array}\right),\quad\mu\in{S}_k.
$$
We then have a bijection of $H$ onto ${S}_k$ defined by $\sigma\mapsto \mu.$
\end{proof}

\section{Congruences on  $\oor$}{\label{conOrthogonal}}
For every equivalence relation $\rho$ on a semigroup $S$ and
$\sigma\in S$, denote by $\rho(\sigma)$ the equivalence class of
$\sigma$ in $\rho$. The following lemma will be useful soon.

\begin{lemma}[\cite{GM}, page 91]\label{congideal}
Let $K$ be an equivalence class of a congruence $\rho$ on a semigroup $S$. If $K$ contains an
ideal of $S$, then $K$ is an ideal of $S$. Moreover, at
most one class of $\rho$ can be an ideal.
\end{lemma}

Let $\sigma,\tau$ be any two elements of $\oor$. Note that $\sigma$ and $\tau$ lie in the same $\H$-class if and only if, by Proposition  \ref{green} (iii), there exists some $k=1,\dots,m,n$ such that
\begin{equation}\label{hclass}
\sigma=\left(
\begin{array}{cccc}
  a_1 & a_2 & \dots & a_k \\
  b_1 & b_2 & \dots & b_k
\end{array}\right)\quad\text{and}\quad
\tau=\left(
\begin{array}{cccc}
  a_1 & a_2 & \dots & a_k \\
  b_{\mu(1)} & b_{\mu(2)} & \dots & b_{\mu(k)}
\end{array}\right)
\end{equation}
 where $\mu\in\w^\p$ if $k=n$, and $\mu\in{S}_k$ if $k=1,\dots,m$.

From now on, for simplicity, the connection (\ref{hclass}) between $\sigma$ and $\tau$ is simply written as $\tau=\mu(\sigma)$.
Using this notation, we now define 4 kinds of congruences on $\oor$.
\begin{definition}
For each normal subgroup $N$ of ${S}_k, k=1,\dots,m-1$ we define an equivalence relation $\equiv_N$ on $\oor$: for $\sigma,\tau\in\oor$,
\begin{enumerate}[{\rm\indent 1)}]
\vspace{-2mm}
  \item if $\rk(\sigma)<k$, then $\sigma\equiv_N\tau$ if and only if
$\rk(\tau)<k$;
\vspace{-2mm}
  \item if $\rk(\sigma)>k$, then $\sigma\equiv_N\tau$ if and only if
$\sigma=\tau$;
\vspace{-2mm}
  \item if $\rk(\sigma)=k$, then
$\sigma\equiv_N\tau$ if and only if $\sigma\H\tau$ and
$\tau=\mu(\sigma)$ for some $\mu\in N$.
\end{enumerate}
\end{definition}

\begin{definition}{\label{def43}}
For each normal subgroup $N$ of ${S}_m$ we define two equivalence relations $\equiv_N^{\rm
I}$ and $\equiv_N^{\rm II}$ on $\oor$: for $\sigma,\tau\in\oor$,
\begin{enumerate}[{\rm\indent 1)}]
\vspace{-2mm}
  \item if $\rk(\sigma)<m$, then $\sigma\equiv_N^{\rm I}\tau$ {\rm (resp.} $\sigma\equiv_N^{\rm II}\tau${\rm )}  if and only if $\rk(\tau)<m$ or $\rk(\tau)=m$ with $\tp(\tau)={\rm II}$ {\rm (resp.} $\tp(\tau)={\rm I}${\rm )};
\vspace{-2mm}
    \item if $\rk(\sigma)=m$ and $\tp(\sigma)={\rm I}$
{\rm (resp.} $\tp(\tau)={\rm II}${\rm )}, then $\sigma\equiv_N^{\rm I}\tau$
{\rm (resp.} $\sigma\equiv_N^{\rm II}\tau${\rm )}  if and only if $\sigma\H\tau$
and $\tau=\mu(\sigma)$ for some $\mu\in N$;
\vspace{-2mm}
    \item if $\rk(\sigma)=m$ and $\tp(\sigma)={\rm II}$ {\rm (resp.} $\tp(\tau)={\rm I}${\rm )}, then $\sigma\equiv_N^{\rm I}\tau$ {\rm (resp.} $\sigma\equiv_N^{\rm II}\tau${\rm )}  if and only if $\rk(\tau)<m$ or $\rk(\tau)=m$ with $\tp(\tau)={\rm II}$ {\rm (resp.} $\tp(\tau)={\rm I}${\rm )};
\vspace{-2mm}
    \item if $\rk(\sigma)=n$, then $\sigma\equiv_N^{\rm I}\tau$ {\rm (resp.} $\sigma\equiv_N^{\rm II}\tau${\rm )}  if and only if $\sigma=\tau$.
\end{enumerate}
\end{definition}

\begin{definition}
For any two normal subgroups $N_1,N_2\lhd{S}_m$ we define the equivalence relation $\equiv_{N_1,N_2}$ on $\oor$: for $\sigma,\tau\in\oor$,
\begin{enumerate}[{\rm\indent 1)}]
\vspace{-2mm}
  \item if $\rk(\sigma)<m$, then $\sigma\equiv_{N_1,N_2}\tau$ if and only
if $\rk(\tau)<m$;
\vspace{-2mm}
  \item if $\rk(\sigma)=m$, then
$\sigma\equiv_{N_1,N_2}\tau$ if and only if $\sigma\H\tau$ and
$\tau=\mu(\sigma)$ for some $\mu\in N_1$ if $\tp(\sigma)={\rm I}$,
and $\mu\in N_2$ if $\tp(\sigma)={\rm II}$;
\vspace{-2mm}
  \item if $\rk(\sigma)=n$, then $\sigma\equiv_{N_1,N_2}\tau$ if and only
if $\sigma=\tau$.
\end{enumerate}
\end{definition}

The above equivalence relations are all well-defined congruences on $\oor$.
By definition neither of the congruences $\equiv_N, \equiv_N^{\rm I}, \equiv_N^{\rm II}$, and $\equiv_{N_1, N_2}$ is uniform. Moreover, neither of the congruences $\equiv_N^{\rm I}$ and $\equiv_N^{\rm II}$ is the identity congruence since there are at least two elements $\sigma, \tau$ of rank $k$ with $k\ge 1$ such that $\sigma\equiv_N^{\rm I}0$ and $\tau\equiv_N^{\rm II}0$. Also, $\equiv_N$ is the identity congruence if and only if $k=1$ and $N={S}_1$; the congruence $\equiv_{N_1, N_2}$ is the identity congruence if and only if $m=1$.

\begin{remark}
Unless $n = 4$, there are no congruences defined above for normal
subgroups of the unit group $\w'$ of $\oor$.
\end{remark}

\begin{lemma}\label{congideal2}
Let $\sigma,\tau\in \oor$ be two elements with $k=\rk (\sigma)>\rk (\tau)$. If $\rho$ is a congruence on $\oor$ and $\sigma\equiv_\rho\tau$, then $\I_k\subseteq\rho (\sigma)$ for $k=1,2,\dots,m-1,n,$ and $\I_m^{\rm I}\subseteq\rho (\sigma)$ or $\I_m^{\rm II}\subseteq\rho (\sigma)$ for
$k=m$.
\end{lemma}
\begin{proof}
Since $\sigma\in\I_k$  (resp. $\sigma\in \I_m^{\rm I}$ or
$\sigma\in\I_m^{\rm II}$), it is enough to show that
$\I_k\subseteq\rho (0)$ (resp. $\I_m^{\rm I}\subseteq\rho (0)$ or
$\I_m^{\rm II}\subseteq\rho (0)$).

Let $l=\rk (\tau)$. We prove that $\I_i\subseteq\rho (0)$ for all $0\le i\leq l$ by induction on
$i$. Clearly the ideal $\I_0\subseteq\rho (0)$, which is then an ideal by
Lemma \ref{congideal}. Now, let $0<j<l$ and assume that for all $i$ with
$i\leq j<l$ we have $\I_i\subseteq\rho (0)$. We will show that
$\I_{j+1}\subseteq\rho (0)$.

Realizing $\im (\sigma)\backslash\im (\tau)\neq \emptyset$, we may choose an element $a_0\in\im (\sigma)\backslash\im (\tau)$ and form a set $A=\{a_0,a_1,\dots,a_j\}\subseteq\im (\sigma)$. It follows immediately that $\rk (\varepsilon_A\sigma)=j+1$, $\rk (\varepsilon_A\tau)\leq j$, and $\varepsilon_A\sigma\equiv_\rho\varepsilon_A\tau$. From the induction hypothesis it follows that $\varepsilon_A\tau\in\rho (0)$, so $\varepsilon_A\sigma\in\rho (0)$. From Corollary \ref{ideal} we conclude that
$\I_{j+1}=\oor\varepsilon_A\sigma\oor \subseteq \rho (0)$.

Hence $\I_l\subseteq\rho (0)$. In particular, $\tau\in\rho (0)$, so $\sigma\in\rho (0)$ since $\sigma\equiv_\rho\tau$. It follows immediately that $\I_{k}=\oor\sigma \oor\subseteq\rho(0)$ for
$k=1,2,\dots,m-1,n$. If $k=m$, we obtain that $\I_m^{\rm I} =\oor\sigma \oor \subseteq\rho (\sigma)=\rho(0)$ when $\sigma$ is of type I, and that $\I_m^{\rm II}=\oor\sigma\oor\subseteq\rho (\sigma)$ when $\sigma$ is of type II.
\end{proof}

If $\rho$ is a congruence on $\oor$, then by Lemma \ref{congideal}
it contains at most one congruence class that is an ideal of
$\oor$. In fact, since $\{0\}$ is an ideal of $\oor$ and $\{0\}\subseteq\rho (0)$, Lemmas
\ref{congideal} and \ref{congideal2} show that the congruence class $\rho (0)$ is an ideal.
\begin{lemma}\label{noidealh}
Let $\rho$ be a congruence on $\oor$. If $\sigma\equiv_\rho\tau$ and $\sigma\notin
\rho (0)$, then $\sigma\H\tau$.
\end{lemma}
\begin{proof}
First, by Lemma \ref{congideal2} we find $\rk (\sigma)=\rk (\tau)$. We claim that $\im (\sigma)=\im (\tau)$. If $\im (\sigma)\ne\im (\tau)$, then $\rk (\sigma)=\rk (\tau)\leq m$ since, otherwise, we have $\rk (\sigma)=\rk (\tau)=n$, which implies $\im (\sigma)=\im (\tau)=\n$. Without loss of generality we may
assume $\im (\sigma)\backslash\im (\tau) \neq\emptyset$. Let
$A=\im (\tau)$. Then
$\varepsilon_A\sigma\equiv_\rho\varepsilon_A\tau=\tau$. Hence
$\sigma\equiv_\rho\varepsilon_A\sigma$. On the other hand,
$\rk (\varepsilon_A\sigma)<\rk (\sigma)$ by the choice of $A$.
Applying Lemma \ref{congideal2} we get $\sigma\in \rho (0)$, which contradicts the assumption.

Let $\im (\sigma)=\{c_1,\dots,c_k\}$. Then we may assume that
$$\sigma=\left (
\begin{array}{cccc}
  a_1 & a_2 & \dots & a_k \\
  c_1 & c_2 & \dots & c_k
\end{array}\right)\qaq\tau=\left (
\begin{array}{cccc}
  b_1 & b_2 & \dots & b_k \\
  c_1 & c_2 & \dots & c_k
\end{array}\right).
$$

We next claim that $\dom (\sigma) = \dom (\tau)$. If not, then we have
$\rk (\sigma)=\rk (\tau)\leq m$, since $\rk (\sigma)=\rk (\tau)=n$
implies $\dom (\sigma)=\dom (\tau)=\n$. Without loss of generality we
may assume that $b_1\notin\dom (\sigma)$, and let
$$\delta=\left (
\begin{array}{cccc}
  c_1 & c_2 & \dots & c_k \\
  b_1 & b_2 & \dots & b_k
\end{array}\right).$$
We have
$\sigma\delta\equiv_\rho\tau\delta$. On the other hand,
$\rk (\tau\delta)=k$ while $\rk (\sigma\delta)<k$ since
$b_1\notin\dom (\sigma)$. Applying Lemma \ref{congideal2}, we get
$\I_k\subseteq\rho (0)$  ($\I_m^{\rm I}\subseteq\rho (0)$ or $\I_m^{\rm
II}\subseteq\rho (0)$ if $k=m$), and hence $\sigma\in \rho (0)$, which is absurd.

The desired result follows from Proposition  \ref{green} iii).
\end{proof}

\begin{lemma}\label{inequality}
Let $\rho$ be a congruence on $\oor$ and let $\sigma\in\oor$ with $\rk(\sigma)=k$. If there exists $\tau\in\oor$ such that $\tau\neq\sigma$ and $\tau\equiv_\rho\sigma$, then
$\rho (0)\supseteq\I_m^{\rm I}$ or $\rho (0)\supseteq\I_m^{\rm II}$ if $k=n$, and $\rho (0)\supseteq\I_{k-1}$ if $0\leq k\leq m$, where we agree that $\I_{-1} = \emptyset$.
\end{lemma}

\begin{proof}
If $\sigma\in\rho (0)$, it is straightforward. Now suppose $\sigma\notin\rho (0)$. We prove the desired result case by case.

Case 1: $0\le k\leq m$. It is trivial for $k=0, 1$. Assume that $2\le k\le m$. By Lemma \ref{noidealh} we can write
$$\sigma=\left (
\begin{array}{cccc}
  a_1 & a_2 & \dots & a_k \\
  b_1 & b_2 & \dots & b_k
\end{array}\right)\quad\text{and}\quad\tau=\left (
\begin{array}{cccc}
  a_{\mu (1)} & a_{\mu (2)} & \dots & a_{\mu (k)} \\
  b_1 & b_2 & \dots & b_k
\end{array}\right)
$$
where $\mu\in{S}_k$. Because $\tau\neq\sigma$, we find $\mu$ is not the identity element. Choose $t$ such that $\mu (t)\neq t$ and let
$A=\im (\sigma)\backslash\{b_t\}$. Then
$\varepsilon_A\sigma\equiv_\rho\varepsilon_A\tau$. Since
$a_t\notin\dom (\varepsilon_A\sigma)$ and
$a_t\in\dom (\varepsilon_A\tau)$, we have
$\dom (\varepsilon_A\sigma)\neq\dom (\varepsilon_A\tau)$. From Lemma
\ref{noidealh} it follows immediately that $\varepsilon_A\sigma\in\rho (0)$.
But then $\rk (\varepsilon_A\sigma)=k-1$ forces $\rho (0)\supseteq
\I_{k-1}$.

 Case 2: $k=n$. Since $\tau\neq\sigma$, we can assume from Lemma \ref{noidealh} that
    $$\sigma=\left (
\begin{array}{cccc}
  a_1 & a_2 & \dots & a_n \\
  1 & 2 & \dots & n
\end{array}\right)\qaq\tau=\left (
\begin{array}{cccc}
  a_{\mu (1)} & a_{\mu (2)} & \dots & a_{\mu (n)} \\
  1 & 2 & \dots & n
\end{array}\right)
$$
where $\mu\in\w^\p$ is different from the identity.

We claim that if $\mu\in\w^\p$ is not the identity element, there is an admissible $m$-subset $I=\{i_1,i_2,\dots,i_m\}$ of $\n$
such that $I\neq\mu (I)=\{\mu (i_1),\mu (i_2),\dots,\mu (i_m)\}$. Indeed, if
$\{1,2,\dots,m\}\neq\{\mu (1),\mu (2),\dots,\mu (m)\}$, then the set
$I=\{1,2,\dots,m\}$, whose type is ${\rm I}$, is the desired admissible $m$-subset in the claim.
If $\{1,2,\dots,m\}=\{\mu (1),\mu (2),\dots,\mu (m)\}$, since $\mu$ is not the identity, there exists $p$ with $1\leq p\leq m$ such that $\mu (p)\neq p$. Consider the
admissible $m$-subset $I=\{1,\dots,p-1,\bar{p},p+1,\dots,m\}$ of type ${\rm II}$. We have $\overline{\mu (p)}=\mu (\bar{p})\in\mu (I)$
and $\overline{\mu (p)}\notin I$. Hence $I\neq\mu (I)$.

Now let $I=\{i_1,i_2,\dots,i_m\}$ be an admissible $m$-subset of $\n$ such that $I\neq\mu (I)$. We obtain
 $\dom (\varepsilon_I\sigma)\neq\dom (\varepsilon_I\tau)$. Hence
$\varepsilon_I\sigma\notin\H (\varepsilon_I\tau)$ by Proposition  \ref{green} iii), and hence $\varepsilon_I\sigma\in\rho (0)$ by Lemma
\ref{noidealh}. But
$\rk (\varepsilon_I\sigma)=m$; it follows that $\rho (0)\supseteq
\I_m^{\rm I}$ or $\rho (0)\supseteq \I_m^{\rm II}$.
\end{proof}

The following result is an immediate consequence of Lemma \ref{inequality}.
\begin{lemma}\label{identity}
Let $\rho$ be a congruence on $\oor$ with $\rho (0)=\I_k$ for some
$k=0,1,\dots,m-1$, and let $\sigma\in \oor$ with $\rk(\sigma)>k+1$. If $\tau\in \oor$, then $\tau\equiv_\rho\sigma$ if and only if
$\tau=\sigma$.
\end{lemma}


We introduce two congruences on the orthogonal rook monoid ${\cal OR}_4$. Note that the unit
group of ${\cal OR}_4$ is $$\w^\p=\left\{\epsilon,~\delta_1= \left (
\begin{array}{cccc}
  1 & 2 & 3 & 4 \\
  2 & 1 & 4 & 3
\end{array}\right),
~
\delta_2=\left (
\begin{array}{cccc}
  1 & 2 & 3 & 4 \\
  3 & 4 & 1 & 2
\end{array}\right),~
\delta_1\delta_2 \right\}~$$
where $\epsilon$ is the identity element of $\w'$.

The first congruence $\equiv_1$ is defined by declaring its equivalency classes:
$\{\epsilon,\delta_1\}$ and $\{\delta_2,\delta_1\delta_2\}$ are two $\equiv_1$ equivalent
classes contained in $\w^\p$, the ideal $\I_2^{\rm II}$ is one
$\equiv_1$ class, and if $\sigma$ is an element of type I of rank 2,
then $\sigma\equiv_1\tau$ if and only if $\sigma\H\tau$.

The second congruence $\equiv_2$ is defined as follows: $\{\epsilon,\delta_2\}$
and $\{\delta_1,\delta_1\delta_2\}$ are two $\equiv_2$ equivalent classes contained in $\w^\p$, the
ideal $\I_2^{\rm I}$ is an $\equiv_2$ class, and if $\sigma$ is an
element of type II of rank 2, then $\sigma\equiv_2\tau$ if and only
if $\sigma\H\tau$.

It is routine to check directly that $\equiv_1$ and
$\equiv_2$ are both congruences on ${\cal OR}_4$.

\begin{lemma}\label{congonw'}
Let $\rho$ be a congruence on $\oor$ such that $\rho (0)=\I_m^{\rm
I}$ or $\rho (0)=\I_m^{\rm II}$. Let $\sigma, \tau\in\oor$ and $\rk(\sigma)=n$.

{\rm i)} If $m\ne 2$, then $\sigma\equiv_\rho\tau$ if and only if
$\sigma=\tau$.

{\rm ii)} If $m = 2$ and there exists $\tau\ne\sigma$ such that
$\tau\equiv_\rho\sigma$, then $\rho$ is $\equiv_1$ if $\rho (0)=\I_m^{\rm II}$, and $\rho$ is $\equiv_2$ if $\rho (0)=\I_m^{\rm I}$.
\end{lemma}
\begin{proof}
Suppose $\sigma\equiv_\rho\tau$ with $\rk (\sigma)=n$. If
$\rk (\tau)<n$, from Lemma \ref{congideal2} it follows that
$\I_n\subseteq\rho(0)$, and hence $\I_n=\rho(0)$, which contradicts $\rho (0)=\I_m^{\rm I}$ or
$\rho (0)=\I_m^{\rm II}$.
Hence
$\rk (\tau)=n$.

  We now prove part i). If $m=1$, then $n=2$ and $\w'=\{1\}$, so the identity is the only element of rank $n$; it is obvious that the statement i) holds.
  
  Now suppose $m\geq 3$.
 If $\rho (0) \subseteq\I_m^{\rm II}$, then every element of type
$\rm I$ in $\oor$ of rank $m$ does not belong to $\rho (0)$. Let $s$
be a number such that $1\leq s\leq n$ and let $\epsilon_I$ be an
idempotent element of $\oor$ where $I=\{a_1,\dots,a_{m-1},\bar{s}\}$
being of type $\rm I$. We have $\sigma\gamma\equiv_\rho\tau\gamma$
and they are both of type $\rm I$. Hence
$\sigma\gamma\H\tau\gamma$ by Lemma \ref{noidealh}. Thus $\sigma
(I) =\tau (I)$. Replacing $\epsilon_I$ by
$\epsilon_J$, where $J=\{a_1,\dots,a_{m-2},\overline{a_{m-1}},s\}$,
we get $\sigma (J) =\tau (J) $.
It follows that
\begin{equation}\label{first}
\{\sigma (\overline{a_{m-1}}) ,\sigma (s) \}=
\{\tau (\overline{a_{m-1}}) ,\tau (s) \}~.
\end{equation}
Similarly, replacing $\epsilon_I$ by $\epsilon_K$, where
$K=\{a_1,\dots,\overline{a_{m-2}},a_{m-1},s\}$, we get $\sigma (K) =\tau (K) $. We find
\begin{equation}\label{second}
\{\sigma (\overline{a_{m-2}}) ,\sigma (s) \}=
\{\tau (\overline{a_{m-2}}) ,\tau (s) \}~.
\end{equation}
It follows from (\ref{first}) and (\ref{second}) that
$\sigma (s) =\tau (s) $, and hence $\sigma=\tau$, since $s$ is arbitrary.
If $\rho (0)=\I_m^{\rm II}$, the proof is similar.

  We next prove part ii). Since only elements of rank 4 are $\rho$ equivalent to elements of rank 4, the restriction
of $\rho$ to $\w^\p$ induces a congruence on $\w^\p$. As $\w^\p$ is
a group, there exists $\eta\in{\cal OR}_4$ such that $\eta\ne\epsilon$
but $\eta\equiv_\rho\epsilon$. Then
$\eta\theta\equiv_\rho\epsilon\theta=\theta$ where $\theta= \left
(\begin{array}{cc}
  1 & 2 \\
  2 & 1
\end{array}\right)$. We obtain
 $\eta\theta\H\theta$ by Lemma \ref{noidealh}. It follows immediately that
$\im (\eta\theta) =\im(\theta)=\{1,2\}$, and thus $\{\eta (1) ,\eta
(2)\}=\{1,2\}$. If $\eta(1)=1$ and $\eta(2)=2$, we would get
$\eta=\epsilon$, which is absurd. Therefore,
$$\eta =\left
(\begin{array}{cccc}
         1 & 2 & 3 & 4 \\
         2 & 1 & 4 & 3
       \end{array}\right) =\delta_1~,
$$
and hence $\epsilon\equiv_\rho\delta_1$ and $\delta_2\equiv_\rho\delta_1\delta_2$.

We claim that for any two elements $\mu,\nu$ of rank 2 with type I if $\mu\H\nu$ then $\mu\equiv_\rho\nu$. Notice that the set of all elements of rank 2 with type I is

\vspace{2mm}
{\footnotesize
$\bigg\{\left(\begin{array}{cc}
         1 & 2  \\
         1 & 2
       \end{array}\right), \left
(\begin{array}{cc}
         1 & 2  \\
         2 & 1
       \end{array}\right)$,
  $\left
(\begin{array}{cc}
         1 & 2  \\
         3 & 4
       \end{array}\right), \left
(\begin{array}{cc}
         1 & 2  \\
         4 & 3
       \end{array}\right)$,
   $\left
(\begin{array}{cc}
         3 & 4  \\
         1 & 2
       \end{array}\right), \left
(\begin{array}{cc}
         3 & 4  \\
         2 & 1
       \end{array}\right)$,
  $\left
(\begin{array}{cc}
         3 & 4  \\
         3 & 4
       \end{array}\right), \left
(\begin{array}{cc}
         3 & 4  \\
         4 & 3
       \end{array}\right)\bigg\}$\\
}\\
We know $\epsilon\alpha\equiv_\rho\delta_1\alpha$ for any element $\alpha$ of rank $2$ with type I
since $\epsilon\equiv_\rho\delta_1$. A direct calculation yields the following table:
\begin{center}

\begin{tabular}{|c|c|}
  \hline

   $\alpha$ & $\epsilon\alpha\equiv_\rho\delta_1\alpha$ \\
  \hline
    $\left
(\begin{array}{cc}
         1 & 2  \\
         1 & 2
       \end{array}\right)$ & $\left
(\begin{array}{cc}
         1 & 2  \\
         1 & 2
       \end{array}\right)\equiv_\rho\left
(\begin{array}{cc}
         1 & 2  \\
         2 & 1
       \end{array}\right)$ \\[8pt]
  $\left
(\begin{array}{cc}
         1 & 2  \\
         3 & 4
       \end{array}\right)$ & $\left
(\begin{array}{cc}
         1 & 2  \\
         3 & 4
       \end{array}\right)\equiv_\rho\left
(\begin{array}{cc}
         1 & 2  \\
         4 & 3
       \end{array}\right)$ \\[8pt]
   $\left
(\begin{array}{cc}
         3 & 4  \\
         1 & 2
       \end{array}\right)$ & $\left
(\begin{array}{cc}
         3 & 4  \\
         1 & 2
       \end{array}\right)\equiv_\rho\left
(\begin{array}{cc}
         3 & 4  \\
         2 & 1
       \end{array}\right)$ \\[8pt]
  $\left
(\begin{array}{cc}
         3 & 4  \\
         3 & 4
       \end{array}\right)$ & $\left
(\begin{array}{cc}
         3 & 4  \\
         3 & 4
       \end{array}\right)\equiv_\rho\left
(\begin{array}{cc}
         3 & 4  \\
         4 & 3
       \end{array}\right)$ \\[8pt]
  \hline
\end{tabular}~.
\end{center}
The claim reads from this table. The given condition that $\rho(0)=\I_m^{\rm II}$ implies $\mu\notin\rho(0)$ where $\mu$ is any element of rank 2 with type I. But Lemma \ref{noidealh} tells us that if $\mu\equiv_\rho\nu$ then $\mu\H\nu$. Hence $\mu\equiv_\rho\nu$ if and only if $\mu\H\nu$. That is, $\rho$ coincides with $\equiv_1$.
\end{proof}

We are now in a position to prove our main result Theorem 1.1, which can be read from the following theorem.
\begin{theorem}\label{finial}
Let $\rho$ be a nonuniform congruence on $\oor$.

{\rm i)} If $\rho (0) =\I_m^{\rm II}$ {\rm (resp. $\rho (0)
=\I_m^{\rm I}$)}  and $n\ne 4$, then $\rho$ coincides with
$\equiv_N^{\rm I}$ {\rm(resp. $\equiv_N^{\rm II}$)}  for some
$N\unlhd S_m$.

{\rm ii)} If $\rho (0) =\I_m^{\rm II}$ {\rm (resp. $\rho (0)
=\I_m^{\rm I}$)}  and $n = 4$, then $\rho$ coincides with $\equiv_1$
{\rm(resp. $\equiv_2$)} or  $\equiv_N^{\rm I}$ {\rm (resp.
$\equiv_N^{\rm II}$)}  for some $N\unlhd S_m$.


{\rm iii)} If $\rho (0) =\I_{m-1}$,  then $\rho$ coincides with
$\equiv_{N_1,N_2}$ for some $N_1,N_2\unlhd S_m$.

{\rm  iv)} If $\rho (0) =\I_{k-1}$ for some $k=1,2,\dots,m-1$, then
$\rho$ coincides with $\equiv_N$ for some normal subgroup $N$ of
$S_k$.
\end{theorem}
\begin{proof}

We first show part i). As $n\ne 4$, it follows from Lemma \ref{congonw'} i) that if
$\sigma\in\oor$ with $\rk(\sigma)=n$ then $\{\sigma\}$ is an
equivalence class. Since $\rho (0) =\I_m^{\rm II}$, all the elements of rank less than or equal to $m-1$ and
those of rank $m$ with type ${\rm II}$ form a single class $\rho
(0)$. We now determine the equivalence classes consisting of
elements of rank $m$ with type ${\rm I}$. Taking
$\sigma,\tau\in\oor$ with $\rk(\sigma) = \rk(\tau)=m$ and
$\tp(\sigma) = \tp(\tau)={\rm I}$, we find from Lemma \ref{noidealh}
that if
 $\sigma\equiv_\rho\tau$ then $\sigma\H\tau$.

We claim that if $\epsilon\in\oor$ with $\epsilon^2=\epsilon$ and
$\rk(\epsilon) = m,\tp(\epsilon) = {\rm I}$, then $\H (\epsilon)
\cong{S}_m$ as groups. Indeed, from Proposition \ref{green} iii) it
follows that $\delta\in\H(\epsilon)$ if and only if $\delta$ and
$\epsilon$ have the same domain, which is equal to same range. Using
$\{c_1, c_2, \ldots, c_m\}$ to denote the domain of $\epsilon$, we find that
$\delta\in\H (\epsilon)$ if and only if
$$
\delta=\left (
\begin{array}{cccc}
  c_1 & c_2 & \dots & c_m \\
  c_{\mu (1) } & c_{\mu (2) } & \dots & c_{\mu (m) }
\end{array}\right),
$$
where $\mu$ is a permutation of $\{1, 2, \ldots, m\}$. Then the
mapping $\delta\mapsto\mu$ yields an isomorphism $\H (\epsilon)
\cong{S}_m$.

It is easily seen that $G (\epsilon) =\{\delta\in\H (\epsilon)
\mid\delta\equiv_\rho\epsilon\}$ is a normal subgroup of $\H
(\epsilon)$, so the image $N$ of $G (\epsilon)$ under the mapping above is a normal subgroup of ${S}_m$.

Let $\sigma, \tau\in\oor$ with $\rk(\sigma) = \rk(\tau) =
m,\tp(\sigma) = \tp(\tau) ={\rm I}$ and $\sigma\mathcal{H}\tau$. Let
\begin{equation*}
\sigma=\left(
\begin{array}{cccc}
  a_1 & a_2 & \dots & a_m \\
  b_1 & b_2 & \dots & b_m
\end{array}\right)\quad\text{and}\quad
\tau=\left(
\begin{array}{cccc}
  a_1 & a_2 & \dots & a_m \\
  b_{\mu(1)} & b_{\mu(2)} & \dots & b_{\mu(m)}
\end{array}\right),
\end{equation*}
where $\mu\in{S}_m$. Define four elements of $\oor$:
$$ \alpha=\left (
\begin{array}{cccc}
  c_1 & c_2 & \dots & c_m \\
  a_1 & a_2 & \dots & a_m
\end{array}\right)\quad\text{and}\quad\beta=\left (
\begin{array}{cccc}
  b_1 & b_2 & \dots & b_m \\
  c_1 & c_2 & \dots & c_m
\end{array}\right) ,$$

\begin{equation}\label{simplecase}
\alpha^\p=\left (
\begin{array}{cccc}
  a_1 & a_2 & \dots & a_m \\
  c_1 & c_2 & \dots & c_m
\end{array}\right) \quad\text{and}\quad\beta^\p=\left (
\begin{array}{cccc}
  c_1 & c_2 & \dots & c_m \\
  b_1 & b_2 & \dots & b_m
\end{array}\right) .
\end{equation}
We obtain $\sigma=\beta^\p\epsilon\alpha^\p,
\,\tau=\beta^\p\delta\alpha^\p, \,
\epsilon=\beta\sigma\alpha,\,\delta=\beta\tau\alpha$. Hence
$\sigma\equiv_\rho\tau$ if and only if $\epsilon\equiv_\rho\delta$
if and only if $\mu\in N$. Thus $\rho$ coincides with $\equiv_N^{\rm
I}$.

We next prove part ii). Note that $n = 4$ and $\rho (0) =\I_m^{\rm II}$. If every equivalent class containing an element of rank $4$ consists of only one element, a similar proof to part i) yields
that $\rho$ coincides with $\equiv^{\rm I}_N$ for some normal
subgroup $N$ of $S_m$. If there exists an equivalent class determined by an element of rank $4$ that contains more than one element, by Lemma \ref{congonw'} ii) we get
$\rho$ coincides with $\equiv_1$. For the case $\rho (0) =\I_m^{\rm I}$ the proof is similar.

In a similar manner, we can prove iii) and iv), and we leave the details to the interested reader.
\end{proof}

\section{Congruences on the symplectic rook monoids $\sr$}\label{conSymplectic}

Note that two elements $\sigma$ and $\tau$ of $\sr$ lie in the same $\H$-class if and only if there exists some $k=1,\dots,m,n$ such that
\begin{equation}\label{hclass}
\sigma=\left(
\begin{array}{cccc}
  a_1 & a_2 & \dots & a_k \\
  b_1 & b_2 & \dots & b_k
\end{array}\right)\quad\text{and}\quad
\tau=\left(
\begin{array}{cccc}
  a_1 & a_2 & \dots & a_k \\
  b_{\mu(1)} & b_{\mu(2)} & \dots & b_{\mu(k)}
\end{array}\right)
\end{equation}
where $\mu\in\w$ if $k=n$, and $\mu\in{S}_k$ if $k=1,\dots,m$.

For each normal subgroup $N$ of $\w$ or of ${S}_k$, $k=1,\dots,m$ we now define a congruence $\equiv_N$ on $\sr$. To make notation consistent, from now on we always specify $k=n$ if $N$ is a normal subgroup of $\w$.

\begin{definition}
    Let $\sigma, \tau\in\sr$.
    \begin{enumerate}[{\indent 1)}]
\vspace{-2mm}
\item If $\rk(\sigma)<k$, then $\sigma\equiv_N\tau$ if and only if
$\rk(\tau)<k$.
\vspace{-2mm}
\item
If $\rk(\sigma)=k$, then $\sigma\equiv_N\tau$ if and only if $\sigma\H\tau$ and $\sigma$ and $\tau$ are given in (\ref{hclass}) for some $\mu\in N$.
\item
\vspace{-2mm}
If $\rk(\sigma)>k$, then $\sigma\equiv_N\tau$ if and only if $\sigma=\tau$.
\end{enumerate}
\end{definition}

The relation $\equiv_N$ is a well-defined congruence on $\sr$. It is the identity congruence if and only if $k=1$.

Our main result about the classification of the congruences on the symplectic rook monoids is stated below.
\begin{theorem}\label{symConResult}
Let $\rho$ be a nonuniform congruence on the symplectic rook monoid $\sr$. Then there is a normal subgroup $N$ of $\w$ or ${S}_k$ for some $k=1,2,\dots,m$ such that $\rho$ coincides with $\equiv_N$.
\end{theorem}
\begin{proof} The proof is a simplified version of Theorem \ref{finial}. We omit its details.
\end{proof}

\vspace{5mm} \noindent Jianqiang Feng

\vspace {-1mm} \noindent College of Mathematics and Information
Science, Hebei University, Baoding, 071002; vonjacky@126.com

\vspace{5mm} \noindent Zhenheng Li

\vspace {-1mm} \noindent Department of
Mathematical Sciences, University of South Carolina Aiken, Aiken SC
29803; Email: zhenhengl@usca.edu

\end{document}